\documentclass[psamsfonts, 10pt]{amsart}
\usepackage[colorlinks=true,
pdfstartview=FitV, linkcolor=cyan, citecolor=magenta,
urlcolor=blue]{hyperref}
\usepackage{amsmath,amsfonts,latexsym,amssymb}
\usepackage{amssymb,amsfonts, amsmath}
\usepackage{mathrsfs}
\usepackage[latin1]{inputenc}
\usepackage[T1]{fontenc}
\usepackage{ae,aecompl}
\usepackage{braket}
\usepackage{comment}
\usepackage{amssymb,amsfonts, amsmath}
 \usepackage{mathtools} 
\usepackage[all,arc]{xy}
\usepackage{enumerate}
\usepackage{mathrsfs}
\usepackage{color}
\usepackage{graphicx}
\usepackage{subfig}
\usepackage{verbatim} 
\usepackage{amssymb,amsfonts, amsmath}
\usepackage[all,arc]{xy}
\usepackage{enumerate}
\usepackage{mathrsfs}
\usepackage{hyperref}
\usepackage{xstring}

\newtheorem{theorem}{Theorem}[section]
\newtheorem{lemma}[theorem]{Lemma}
\newtheorem{proposition}[theorem]{Proposition}
\newtheorem{corollary}[theorem]{Corollary}

\newtheorem{Example}[theorem]{Example}
\newtheorem{fact}[theorem]{Fact}

\long\def\@savemarbox#1#2{\global\setbox#1\vtop{\hsize\marginparwidth 
  \@parboxrestore\tiny\raggedright #2}}
\title{On Borel Anosov representations in even dimensions}
\author{Konstantinos Tsouvalas}
\address{Department of Mathematics, University of Michigan, 530 Church
Street, Ann Arbor, MI 48109, USA \newline \small{\emph{E-mail address}:  \textsf{tsouvkon@umich.edu}}}

\begin{document}
\maketitle
\begin{abstract} We prove that a word hyperbolic group which admits a $P_{2q+1}$-Anosov representation into $\mathsf{PGL}(4q+2, \mathbb{R})$ contains a finite-index subgroup which is either free or a surface group. As a consequence, we give an affirmative answer to Sambarino's question for Borel Anosov representations into $\mathsf{SL}(4q+2,\mathbb{R})$. \end{abstract}
\section{Introduction}
In this note, we address the following question of Andr\'es Sambarino and provide a positive answer when $d=4q+2$ for some $q \in \mathbb{N}$.
\medskip

\noindent \textbf{Sambarino's Question:} \emph{Suppose that $\Gamma$ is a torsion free word hyperbolic group which admits a Borel Anosov representation into $\mathsf{SL}(d,\mathbb{R})$. Is $\Gamma$ necessarily free or a surface group?.}
\medskip

\par Anosov representations of word hyperbolic groups into real semisimple Lie groups were introduced by Labourie \cite{labourie-invent} in his study of the Hitchin component. They are discrete subgroups of real reductive Lie groups which generalize convex cocompact subgroups of rank one Lie groups. A representation $\rho:\Gamma \rightarrow \mathsf{GL}(d,\mathbb{R})$ is called \emph{$P_{k}$-Anosov}, where $1\leqslant k\leqslant \frac{d}{2}$, if it is Anosov with respect to the pair of opposite parabolic subgroups of $\mathsf{GL}(d,\mathbb{R})$ defined as the stabilizers of a $k$-plane and a complementary $(d-k)$-plane (see subsection \ref{definition}). The representation $\rho$ is called \emph{Borel Anosov} if $\rho$ is $P_k$-Anosov for every $k$. Labourie in \cite{labourie-invent} proved that every Hitchin representation into $\mathsf{PSL}(d,\mathbb{R})$ is irreducible and admits a lift into $\mathsf{GL}(d,\mathbb{R})$ which is Borel Anosov. The only known examples of Borel Anosov representations are constructed from representations of free or surface groups. By a surface group we mean the fundamental group of a closed surface of negative Euler characteristic. Hitchin representations are the only known examples of Borel Anosov representations of surface groups in even dimensions. In all odd dimensions, Barbot's construction \cite{Barbot} can be used to produce reducible examples.

A positive answer to Sambarino's question was given in \cite{CT} for $d=3$ or $4$. By using results of Benoist in \cite{benoist-divisible0,benoist-divisible3}, we prove that a torsion free word hyperbolic group admitting a $P_{2q+1}$-Anosov representation into $\mathsf{GL}(4q+2,\mathbb{R})$ has to be either free or a surface group. Moreover, by using Wilton's result \cite{Wilton} on the existence of quasiconvex surface groups or rigid subgroups in one ended-word hyperbolic groups and a theorem of Kapovich-Leeb-Porti in \cite{KLP2} (see also \cite[Theorem 6]{KLP3}), we prove the following stronger statement: 

\begin{theorem} \label{maintheorem} Let $\Gamma$ be a word hyperbolic group and $\rho:\Gamma \rightarrow \mathsf{GL}(4q+2, \mathbb{R})$ a representation. Suppose that there exists a continuous, $\rho$-equivariant dynamics preserving map $\xi:\partial_{\infty}\Gamma \rightarrow \mathsf{Gr}_{2q+1}(\mathbb{R}^{4q+2})$. Then $\Gamma$ is virtually free or virtually a surface group. \end{theorem}

\noindent The group $\Gamma$ is virtually free (resp. a surface group) if it contains a finite-index subgroup which is free (resp. a surface group). The map $\xi$ is called dynamics preserving whenever $\gamma \in \Gamma$ is an infinite order element, $\rho(\gamma)$ is $P_k$-proximal and $\xi(\gamma^{+})$ is its attracting fixed point in $\mathsf{Gr}_{2q+1}(\mathbb{R}^{4q+2})$. An analogue of Theorem \ref{maintheorem} does not hold in dimensions which are multiples of $4$, see Section \ref{examples}.
\medskip

\begin{corollary}\label{cor1} Let $\mathsf{G}_{4q+2}$ be either $\mathsf{GL}(4q+2,\mathbb{R})$ or $\mathsf{PGL}(4q+2,\mathbb{R})$. If $\Gamma$ is a word hyperbolic group and $\rho:\Gamma \rightarrow \mathsf{G}_{4q+2}$ is a $P_{2q+1}$-Anosov representation, then $\Gamma$ is virtually free or virtually a surface group. \end{corollary}

Let $\tau^{+}_{k}:\mathsf{Gr}_{k}(\mathbb{R}^d) \rightarrow \mathbb{P}(\wedge^{k}\mathbb{R}^d)$ be the Pl${\textup{\"u}}$cker embedding (see subsection \ref{proximality}). By using the connectedness properties of the boundary of a rigid hyperbolic group with the methods of the proof of Theorem \ref{maintheorem} we have:

\medskip

\begin{corollary} \label{rigid} Let $\Gamma$ be a torsion free rigid word hyperbolic group and $\rho:\Gamma \rightarrow \mathsf{GL}(4q+2,\mathbb{R})$ a representation. Suppose there exists a continuous $\rho$-equivariant map $\xi:\partial_{\infty}\Gamma \rightarrow  \mathsf{Gr}_{2q+1}(\mathbb{R}^{4q+2})$. Then the map $\xi$ is nowhere dynamics preserving and $\tau_{2q+1}^{+}\circ \xi$ is not spanning. \end{corollary}

\noindent The map $\xi$ is called nowhere dynamics preserving if for every infinite order element $\gamma \in \Gamma$, the restriction of $\xi$ on $\{\gamma^{-},\gamma^{+}\}$ is not dynamics preserving.

\vspace{0.3cm}

\noindent \textbf{Acknowledgements.} I would like to thank my advisor Richard Canary for his support and many useful comments on earlier versions of this paper and Andr\'es Sambarino for his question. This work was partially supported by grants DMS-1564362 and DMS-1906441 from the National Science Foundation.

\section{Background}
In this section, we provide some background on proximality, define Anosov representations and state Benoist's results that we are going to use for the proof of the main theorem.

\subsection{Proximality.} \label{proximality} Let $d \geqslant 2$ and $e_1,..,e_d$ be the canonical basis of $\mathbb{R}^d$. For an element $g \in \mathsf{GL}(d,\mathbb{R})$ we denote by $\lambda_{1}(g)\geqslant \lambda_{2}(g) \geqslant ... \geqslant \lambda_{d}(g)$ the moduli of its eigenvalues. For $1 \leqslant k \leqslant \frac{d}{2}$, we denote by $P_{k}$ the stabilizer of the plane $\langle e_1,..,e_k \rangle$ and by $P_{k}^{-}$ the stabilizer of the complementary $(d-k)$-plane $\langle e_{k+1},...,e_d \rangle$. The Grassmannian of $k$-planes, $\mathsf{Gr}_{k}(\mathbb{R}^d)$ is identified with the quotient manifold $\mathsf{GL}(d,\mathbb{R})/P_{k}$. Similarly $ \mathsf{Gr}_{d-k}(\mathbb{R}^d)$ is identified with $\mathsf{GL}(d,\mathbb{R})/P_{k}^{-}$. A pair of planes $(V^{+},V^{-}) \in  \mathsf{Gr}_{k}(\mathbb{R}^d) \times  \mathsf{Gr}_{d-k}(\mathbb{R}^d)$ is \emph{transverse} if there exists $h \in \mathsf{GL}(d,\mathbb{R})$ such that $V^{+}=h\langle e_1,...,e_{k}\rangle$ and $V^{-}=h\langle e_{k+1},..,e_{d}\rangle$. An element $g \in \mathsf{GL}(d, \mathbb{R})$ is called $P_{k}$-\emph{proximal} if $\lambda_{k}(g)>\lambda_{k+1}(g)$. Equivalently, $g$ has two fixed points $x_{g}^{+}\in  \mathsf{Gr}_{k}(\mathbb{R}^d)$ and \hbox{$V_{g}^{-} \in  \mathsf{Gr}_{d-k}(\mathbb{R}^d)$} such that the pair $(x_{g}^{+},V_{g}^{-})$ is transverse and for every $k$-plane $V_{0}$ transverse to $V_{g}^{-}$, we have $\lim_{n}g^nV_{0}=x_{g}^{+}$. The element $g$ is called $P_{k}$-\emph{biproximal} if $g$ and $g^{-1}$ are $P_{k}$-proximal. We denote by $x_{g}^{-}$ the attracting fixed point of $g^{-1}$ in $ \mathsf{Gr}_{k}(\mathbb{R}^d)$.
For $k=1$, a $P_1$-proximal element $g \in \mathsf{GL}(d, \mathbb{R})$ in $\mathbb{P}(\mathbb{R}^d)$ has a unique eigenvalue, $\ell_1(g)$, of maximum modulus with multiplicity exactly one. The repelling hyperplane of $g$ is denoted by $V_{g}^{-}$. The matrix $g$ is called $P_1$-\emph{positively proximal} if $\ell_1(g)>0$.
\par The Pl${\textup{\"u}}$cker embeddings $\tau_{k}^{+}:\mathsf{Gr}_{k}(\mathbb{R}^d) \rightarrow \mathbb{P}(\wedge^k \mathbb{R}^d)$ and $\tau_{k}^{-}:\mathsf{Gr}_{d-k}(\mathbb{R}^d) \rightarrow \mathsf{Gr}_{d_k-1}(\wedge^{k}\mathbb{R}^d)$, $d_k=\binom{d}{k}$, where $$\tau_{k}^{+}(gP_{k})=[ge_1\wedge...\wedge ge_k] \ \ \ \tau^{-}_{k}(gP_{k}^{-})=[\wedge^{k}g (e_1 \wedge ... \wedge e_{k})^{\perp}]$$ are embeddings and an element $g$ is $P_k$-proximal if and only if $\tau_{k}^{+}(g)$ is $P_1$-proximal (see also \cite[Proposition 3.3]{GGKW} for more details).
\medskip

From now, unless specified, proximal (resp. positively proximal) will refer to $P_1$-proximality (resp. positive $P_1$-proximality) in the projective space.

\subsection{Dynamics preserving maps.} Let $\Gamma$ be a word hyperbolic group and denote by $\partial_{\infty}\Gamma$ its Gromov boundary. Every infinite order element $\gamma \in \Gamma$ has exactly two fixed points $\gamma^{+}$ and $\gamma^{-}$ on $\partial_{\infty}\Gamma$ called the attracting and repelling fixed points of $\gamma$ respectively. Let $\rho:\Gamma \rightarrow \mathsf{GL}(d, \mathbb{R})$ be a representation and $1 \leqslant k \leqslant d-1$. Suppose there exists a continuous $\rho$-equivariant map $\xi:\partial_{\infty}\Gamma \rightarrow \mathsf{Gr}_{k}(\mathbb{R}^d)$. The map $\xi$ is called \emph{dynamics preserving} if for every element $\gamma \in \Gamma$ of infinite order, $\rho(\gamma)$ is $P_k$-proximal and $\xi(\gamma^{+})=x_{\rho(\gamma)}^{+}$. The map $\xi$ is called \emph{nowhere dynamics preserving} if for every $\gamma \in \Gamma$ the restriction of $\xi$ on $\partial_{\infty}\langle \gamma \rangle=\{\gamma^{-},\gamma^{+}\}$ is not dynamics preserving.
\medskip

\subsection{Anosov representations.} \label{definition} The dynamical definition of Anosov representations (see \cite{GW, labourie-invent}) involves the geodesic flow of a word hyperbolic group. Characterizations of Anosov representations into real reductive Lie groups, without involving flow spaces, have been established in several papers, see \cite{BPS, GGKW, KLP1, KP}. Here we define Anosov representations by using a characterization of Kapovich-Leeb-Porti in \cite{KLP1} and Bochi-Potrie-Sambarino \cite{BPS}. For a finitely generated group $\Gamma$ we always fix a left-invariant word metric and for $\gamma \in \Gamma$, $|\gamma|_{\Gamma}$ is the distance of $\gamma$ from the identity element of $\Gamma$. For an element $g \in \mathsf{GL}(d,\mathbb{R})$ let $\sigma_{1}(g) \geqslant \sigma_{2}(g) \geqslant... \geqslant \sigma_{d}(g)$ be the singular values of $g$. Recall that for each $i$, $\sigma_{i}(g)=\sqrt{\lambda_{i}(gg^t)}$, where $g^t$ is the transpose of $g$. Notice that for an element $[h] \in \mathsf{PGL}(d,\mathbb{R})$ the ratio $\frac{\sigma_{i}(h)}{\sigma_{i+1}(h)}$ does not depend on the choice of the representative $h \in \mathsf{GL}(d,\mathbb{R})$.
\par Let $\mathsf{G_d}$ be either $\mathsf{GL}(d,\mathbb{R})$ or $\mathsf{PGL}(d,\mathbb{R})$, $\rho:\Gamma \rightarrow \mathsf{G}_d$ a representation and $1 \leqslant k \leqslant \frac{d}{2}$. Then $\rho$ is $P_{k}$-Anosov if and only if there exist $C,\alpha >0$ such that $$ \frac{\sigma_{k}(\rho(\gamma))}{\sigma_{k+1}(\rho(\gamma))} \geqslant Ce^{\alpha |\gamma|_{\Gamma}}$$ for every $\gamma \in \Gamma$.
\par It is clear from the previous definition that for every quasiconvex subgroup $H$ of $\Gamma$ the restriction $\rho|_{H}$ is $P_{k}$-Anosov. The following theorem summarizes some of the properties of Anosov representations. 
\medskip

\begin{theorem}\label{mainproperties} \cite{GW,labourie-invent} Let $\mathsf{G}_d$ be either $\mathsf{GL}(d,\mathbb{R})$ or $\mathsf{PGL}(d,\mathbb{R})$ and $\Gamma$ be a word hyperbolic group. Suppose $1 \leqslant k \leqslant \frac{d}{2}$ and $\rho:\Gamma \rightarrow \mathsf{G}_{d}$ is a $P_{k}$-Anosov representation. Then:
\medskip

\noindent \textup{(i)} $\rho$ is a quasi-isometric embedding, i.e. there exist constants $A,C>0$ such that for every $\gamma \in \Gamma$ $$\frac{1}{C}|\gamma|_{\Gamma}-A \leqslant \log \frac{\sigma_1(\rho(\gamma))}{\sigma_{d}(\rho(\gamma))} \leqslant C|\gamma|_{\Gamma}+A$$ 
\medskip

\noindent \textup{(ii)} There exist continuous $\rho$-equivariant maps $$\xi_{\rho}^{k}:\partial_{\infty}\Gamma \rightarrow  \mathsf{Gr}_{k}(\mathbb{R}^d) \ \ \ \xi_{\rho}^{d-k}:\partial_{\infty}\Gamma \rightarrow  \mathsf{Gr}_{d-k}(\mathbb{R}^d)$$ which are dynamics preserving and for distinct points $x,y \in \partial_{\infty}\Gamma$ the pair $(\xi_{\rho}^{k}(x),\xi_{\rho}^{d-k}(y))$ is transverse.
\medskip

\noindent \textup{(iii)} The set of $P_{k}$-Anosov representations of $\Gamma$ in $\mathsf{G}_d$ is open in $\textup{Hom}(\Gamma,\mathsf{G}_d)$.
\end{theorem}
\medskip

Notice that by the previous definition, the representation $\rho$ is $P_{k}$-Anosov if and only if $\wedge^k \rho$ is \hbox{$P_{1}$-Anosov}. The Anosov limit maps of $\wedge^k \rho$ are $\tau_{d,k}^{+}\circ \xi_{\rho}^{k}$ and $\tau_{d,k}^{-}\circ \xi_{\rho}^{d-k}$.
\medskip

We also need the following fact which implies the continuity of the first eigenvalue among $P_1$-Anosov representations.

\begin{fact} \label{continuity} Let $\{A_{t}\}_{t \in [0,1]}$ be a continuous family of proximal elements of $\mathsf{GL}(d,\mathbb{R})$. Then, the function $t \mapsto \ell_1(A_t)$ is continuous.\end{fact}

\begin{proof} 

The conclusion follows immediately from the continuity of the characteristic polynomial of matrices. \end{proof}

\subsection{The work of Benoist.} We summarize here some results that we use from \cite{benoist-divisible0} and \cite{benoist-divisible3}. An open cone $C \subset \mathbb{R}^d$ is called \emph{properly convex} if it does not contain an affine line. A domain $\Omega \subset \mathbb{P}(\mathbb{R}^d)$ is called \emph{properly convex} if it is contained in some affine chart of $\mathbb{P}(\mathbb{R}^d)$ in which $\Omega$ is bounded and convex. An element $g \in \mathsf{GL}(d, \mathbb{R})$ is called positively semi-proximal if $\lambda_1(g)$ is an eigenvalue of $g$. A subgroup $\Gamma$ of $\mathsf{GL}(d, \mathbb{R})$ is called \emph{positively proximal} if it contains a proximal element and every proximal element of $\Gamma$ is positively proximal. 

\begin{lemma}\cite[Lemma 3.2]{benoist-divisible3} \label{cone} Let $\Gamma$ be a subgroup of $\mathsf{GL}(d, \mathbb{R})$ which preserves a properly convex open cone $C$ in $\mathbb{R}^d$. Then every $\gamma \in \Gamma$ is positively semi-proximal. In particular, every proximal element $\gamma \in \Gamma$ is positively proximal. \end{lemma}

Benoist characterized irreducible subgroups of $\mathsf{GL}(d,\mathbb{R})$ which preserve a properly convex cone in $\mathbb{R}^d$ as follows:

\begin{theorem}\cite[Proposition 1.1]{benoist-divisible0} \label{positivity3} Let $\Gamma$ be an irreducible subgroup of $\mathsf{GL}(d, \mathbb{R})$. Then $\Gamma$ preserves a properly convex open cone $C$ in $\mathbb{R}^d$ if and only if $\Gamma$ is positively proximal.\end{theorem}

We also have the following fact for subgroups of $\mathsf{GL}(d,\mathbb{R})$ which preserve properly convex domains in $\mathbb{P}(\mathbb{R}^d)$:

\begin{fact} \label{index2} \normalfont  Let $\Gamma$ be a subgroup of $\mathsf{GL}(d, \mathbb{R})$ which preserves a properly convex domain $\Omega \subset \mathbb{P}(\mathbb{R}^d)$. There exists a representation $\widetilde{\iota}:\Gamma \rightarrow \mathsf{GL}(d, \mathbb{R})$ and a group homomorphism $\varepsilon:\Gamma \rightarrow \mathbb{Z}/2$ such that: $\widetilde{\iota}(\gamma)=(-1)^{\varepsilon(\gamma)}\gamma$ for every $\gamma \in \Gamma$ and $\widetilde{\iota}(\Gamma)$ preserves a properly convex open cone $C$ lifting $\Omega$. Thus, if $\Gamma$ is also finitely generated the group $\Gamma_2:=\bigcap\{H:[\Gamma:H]\leqslant2\}$ has finite-index in $\Gamma$ and preserves the properly convex cone $C$.\end{fact}

We will also use the following fact:

\begin{proposition} \label{discrete} Let $\Gamma$ be a word hyperbolic group and $\rho:\Gamma \rightarrow \mathsf{GL}(d,\mathbb{R})$ be a representation. If there exists a continuous $\rho$-equivariant non-constant map $\xi:\partial_{\infty}\Gamma \rightarrow \mathbb{P}(\mathbb{R}^d)$, then $\rho$ is discrete and $\textup{ker}(\rho)$ is finite.  \end{proposition}

\begin{proof}
Assume that there exists an infinite sequence $(\gamma_n)_{n \in \mathbb{N}}$ of elements of $\Gamma$ with $\lim_{n}\rho(\gamma_n)=I_{d}$. The group $\Gamma$ acts on $\partial_{\infty}\Gamma$ as a convergence group, hence up to subsequence, there exists $\eta, \eta' \in \partial_{\infty}\Gamma$ with $\lim_{n}\gamma_n x=\eta$ for $x \neq \eta'$ and $\xi(x)=\xi(\eta)$, $x \neq \eta'$. Since $\partial_{\infty}\Gamma$ is perfect, $\xi$ has to be constant, a contradiction. In particular, $\textup{ker}(\rho)$ is a torsion subgroup of $\Gamma$, hence finite.\end{proof}

Let $F_k$ be the free group on $k$ generators. We close this section with the following proposition which follows by the work of Breuillard-Green-Guralnick-Tao (see \cite[Theorem 4.1]{BGGT}):

\begin{proposition}\cite{BGGT} \label{density} The set of Zariski dense representations from $F_2$ in $\mathsf{SL}(d, \mathbb{R})$ is dense in the representation variety $\textup{Hom}(F_k, \mathsf{SL}(d, \mathbb{R}))$. \end{proposition}

\section {Proof of the main result}
In this section we give the proof of Theorem \ref{maintheorem}. First, we need the following lemma which is proved using a theorem of Kapovich-Leeb-Porti \cite{KLP2} (see also \cite{CLS}).

\begin{lemma} \label{free} Let $\Gamma$ be a torsion free non-elementary word hyperbolic group and $\rho:\Gamma \rightarrow \mathsf{GL}(d, \mathbb{R})$ be a representation which admits a continuous $\rho$-equivariant map $\xi:\partial_{\infty}\Gamma \rightarrow \mathbb{P}(\mathbb{R}^d)$. Suppose there exists $\gamma \in \Gamma$ such that $\rho(\gamma)$ is biproximal, $\xi(\gamma^{+})=x_{\rho(\gamma)}^{+}$ and $\xi(\gamma^{-})=x_{\rho(\gamma)}^{-}$. Then, there exist $a,b \in \Gamma$ such that $\langle a, b\rangle$ is a free quasiconvex subgroup of $\Gamma$ of rank $2$ and the restricted representation $\rho: \langle a,b \rangle \rightarrow \mathsf{GL}(d, \mathbb{R})$ is $P_1$-Anosov with Anosov limit map $\xi$. \end{lemma}

\begin{proof} By Proposition \ref{discrete}, the representation $\rho$ is discrete and faithful. Let $t \in \Gamma$ be an infinite order element such that $\{\gamma^{+},\gamma^{-}\}\cap \{t^{+},t^{-}\}$ is empty. Up to conjugating $\rho$ we may assume that $x_{\rho(\gamma)}^{+}=[e_1], x_{\rho(\gamma^{-1})}^{+}=[e_d]$ and $V_{\rho(\gamma)}^{-}=\langle e_2,...,e_d \rangle$, $V_{\rho(\gamma^{-1})}^{-}=\langle e_1,...,e_{d-1} \rangle$. Then we notice that $$\rho(t^{\pm1})x_{\rho(\gamma)}^{+} \notin \mathbb{P}(V_{\rho(\gamma)}^{-})\cup \mathbb{P}(V_{\rho(\gamma^{-1})}^{-}) \ \ \ \textup{and} \ \ \ \rho(t^{\pm 1})x_{\rho(\gamma)}^{-} \notin \mathbb{P}(V_{\rho(\gamma)}^{-}) \cup \mathbb{P}(V_{\rho(\gamma^{-1})}^{-}) $$ For example, suppose that $\rho(t)x_{\rho(\gamma)}^{+} \in \mathbb{P}(V_{\rho(\gamma)}^{-})$, then $\lim_{n}\rho(\gamma^n)\rho(t)x_{\rho(\gamma)}^{+}=\lim_{n}\xi(\gamma^{n}t\gamma^{+})=\xi(\gamma^{+})=[e_1]$ has to be in $\mathbb{P}(V_{\rho(\gamma)}^{-})$, a contradiction. Since, $\lim_{n}\gamma^n t^{-1}\gamma^{+}=\gamma^{+}$ we have $\lim_{n}\rho(\gamma^n t^{-1})\xi(\gamma^{+})=x_{\rho(\gamma)}^{+}$ and $\rho(t^{-1})x_{\rho(\gamma^{-1})}^{+} \notin \mathbb{P}(V_{\rho(\gamma)}^{-})$. Then, by \cite[Theorem 7.40]{KLP2} (see also \cite[Theorem A2]{CLS}), there exists $N>0$ such that the group $H=\langle \gamma^N,t\gamma^nt^{-1}\rangle$ is a free group of rank $2$ and the restriction $\rho|_{H}$ is $P_1$-Anosov. The restriction $\rho|_{H}$ is also a quasi-isometric embedding hence $H$ is a quasiconvex subgroup of $\Gamma$ and its Anosov limit map is the restriction of $\xi$ on $\partial_{\infty}H$ considered as a subset of $\partial_{\infty}\Gamma$. \end{proof}

Recall that for a finitely generated group $\Gamma$, $\Gamma_2$ is defined to be the intersection of all finite-index subgroups of $\Gamma$ of index at most $2$. 

\begin{lemma} \label{positive1} Let $\Gamma$ be a torsion free one-ended word hyperbolic group and $\rho:\Gamma \ast \mathbb{Z} \rightarrow \mathsf{GL}(d, \mathbb{R})$ be a representation which admits a $\rho$-equivariant continuous map $\xi:\partial_{\infty}(\Gamma \ast \mathbb{Z}) \rightarrow \mathbb{P}(\mathbb{R}^d)$. Suppose that $\delta \in \Gamma_2$ is a non-trivial element such that $\rho(\delta)$ is biproximal and $\xi(\delta^{+})=x_{\rho(\delta)}^{+}$ and $\xi(\delta^{-})=x_{\rho(\delta)}^{-}$. Then $\rho(\delta)$ is positively proximal. \end{lemma}

\begin{proof} Let $s$ be a generator of the free cyclic factor, $t=s\delta s^{-1} \in \Gamma$ and notice that $\rho(t)$ is proximal with $\rho(s)x_{\rho(\delta)}^{+}=x_{\rho(t)}^{+}=\xi(t^{+})$ and $t^{\pm} \notin \partial_{\infty}\Gamma$. If $x \in \partial_{\infty}\Gamma$, $\lim_{n}\rho(t^n)\xi(x)=\lim_{n}\xi(t^n x)=\xi(t^{+})$. Since $\rho(t)$ preserves $V_{\rho(t)}^{-}$ and $\lim_{n}t^{n}x=t^{+}$, $\xi(x)$ cannot lie in $\mathbb{P}(V_{\rho(t)}^{-})$. It follows that $\xi(\partial_{\infty}\Gamma)$ lies in the affine chart $\mathbb{P}(\mathbb{R}^d)-\mathbb{P}(V_{\rho(t)}^{-})$. Let $V=\langle \xi(\partial_{\infty}\Gamma) \rangle$ and we consider the representation $\rho':\Gamma \rightarrow \mathsf{GL}(V)$ where $\rho'(\gamma)=\rho|_{V}(\gamma)$, $\gamma \in \Gamma$. The map $\xi$ is not constant, hence $\rho'$ is discrete and faithful. The map \hbox{$\xi: \partial_{\infty}\Gamma \rightarrow \mathbb{P}(V)$} is $\rho'$-equivariant, $\rho'(\delta)$ is proximal with attracting fixed point $\xi(\delta^{+})$ and $\ell_1(\rho(\delta))=\ell_1(\rho'(\delta))$. 
\par Then we notice that $\xi(\partial_{\infty}\Gamma)$ also lies in the affine chart $A=\mathbb{P}(V)-\mathbb{P}(V\cap V_{\rho(t)}^{-})$ of $\mathbb{P}(V)$. Since $\Gamma$ is one-ended, $\partial_{\infty}\Gamma$ and $\xi(\partial_{\infty}\Gamma)$ are connected. The convex hull of $\xi(\partial_{\infty}\Gamma)$ in $A$, say $\mathcal{C}$, is bounded and convex in $A$ and has non-empty interior since $\xi(\partial_{\infty}\Gamma)$ spans $V$. Then $\rho'(\Gamma)$ preserves $\xi(\partial_{\infty}\Gamma)$ and by \cite[Proposition 2.8]{CT} it also preserves $\mathcal{C}$. It follows that $\rho'(\Gamma)$ preserves the non-empty properly convex set $\Omega=\textup{Int}(\mathcal{C}) \subset \mathbb{P}(V)$. Fact \ref{index2} shows that there exists a representation $\widetilde{\rho}':\Gamma \rightarrow \mathsf{GL}(V)$ which preserves a properly convex cone $C \subset V$ and $\rho'(\gamma)=\widetilde{\rho}'(\gamma)$ for every $\gamma \in \Gamma_2$. By Lemma \ref{cone}, $\rho(\delta)$ is positively proximal in $\mathbb{P}(V)$ and hence in $\mathbb{P}(\mathbb{R}^d)$.\end{proof}

A torsion free word hyperbolic group $\Gamma$ is called \emph{rigid} if it does not admit a non-trivial splitting over a cyclic subgroup. For example, the fundamental group of a closed negatively curved Riemannian manifold of dimension at least $3$ is rigid. By a theorem of Bowditch \cite{Bowditch} the Gromov boundary $\partial_{\infty}\Gamma$ of a rigid hyperbolic group $\Gamma$ does not contain local cut points. 

\begin{lemma} \label{positive2} Let $\Gamma$ be a torsion free rigid one-ended word hyperbolic group. Let $\rho:\Gamma \rightarrow \mathsf{GL}(d, \mathbb{R})$ be a representation which admits a continuous $\rho$-equivariant map $\xi:\partial_{\infty}\Gamma \rightarrow \mathbb{P}(\mathbb{R}^d)$. Suppose that $\delta \in \Gamma_2$ is a non-trivial element such that $\rho(\delta)$ is biproximal and $\xi(\delta^{+})=x_{\rho(\delta)}^{+}$ and $\xi(\delta^{-})=x_{\rho(\delta)}^{-}$. Then $\rho(\delta)$ is positively proximal. \end{lemma}

\begin{proof} Since $\partial_{\infty}\Gamma$ does not have any local cut points, the set $\partial_{\infty}\Gamma-\{\delta^{+},\delta^{-}\}$ is connected. For \hbox{$x \neq \delta^{+},\delta^{-}$} we have that $\lim_{n}\delta^{\pm n} x=\delta^{\pm}$ and, as in Lemma \ref{positive1}, the conected set $\xi \big(\partial_{\infty}\Gamma-\{\delta^{+},\delta^{-}\}\big)$ is contained in \hbox{$\mathbb{P}(\mathbb{R}^d)-\mathbb{P}(V_{\rho(\delta)}^{-})\cup\mathbb{P}(V_{\rho(\delta^{-1})}^{-})$}. Note that the two $(d-1)$-planes $V_{\rho(\delta)}^{-}$ and $V_{\rho(\delta^{-1})}^{-}$ are distinct, hence by the connectedness of $\partial_{\infty}\Gamma-\{\delta^{+},\delta^{-}\}$ we can find a hyperplane $V_0$ such that $\xi(\partial_{\infty}\Gamma)$ is contained in $\mathbb{P}(\mathbb{R}^d)-\mathbb{P}(V_0)$. Then we consider the restriction $\rho':\Gamma \rightarrow \mathsf{GL}(V)$, $V=\langle \xi(\partial_{\infty}\Gamma)\rangle$, whose image preserves the compact connected subset $\xi(\partial_{\infty}\Gamma)$ of the affine chart $\mathbb{P}(V)-\mathbb{P}(V\cap V_0)$ of $\mathbb{P}(V)$. The element $\rho'(\gamma)$ is proximal in $\mathbb{P}(V)$ and $\ell_1(\rho(\gamma))=\ell_1(\rho'(\gamma))$. We similarly conclude that $\rho'(\Gamma)$ preserves a properly convex domain $\Omega$ of $\mathbb{P}(V)$. Again, Fact \ref{index2} guarantees that $\rho'(\Gamma_2)$ preserves a properly convex cone of $V$ and $\ell_1(\rho'(\delta))>0$. \end{proof}

Now we combine the previous results to prove Theorem \ref{maintheorem}.

\medskip \noindent {\bf Theorem \ref{maintheorem}:} {\em Let $\Gamma$ be a word hyperbolic group and $\rho:\Gamma \rightarrow \mathsf{GL}(4q+2, \mathbb{R})$ a representation. Suppose that there exists a continuous, $\rho$-equivariant dynamics preserving map $\xi:\partial_{\infty}\Gamma \rightarrow  \mathsf{Gr}_{2q+1}(\mathbb{R}^{4q+2})$. Then $\Gamma$ is virtually free or virtually a surface group.}

\begin{proof}  We first assume that $\Gamma$ is a torsion free hyperbolic group. By Proposition \ref{discrete}, $\rho$ is faithful and we may assume that $\rho(\Gamma)$ is a subgroup of $\mathsf{SL}(4q+2, \mathbb{R})$. If not, we replace $\rho$ with the representation $\hat{\rho}:\Gamma \rightarrow \mathsf{SL}^{\pm}(n,\mathbb{R})$, $\hat{\rho}(\gamma)=|\textup{det}(\rho(\gamma))|^{-1/(4q+2)}\rho(\gamma)$ and $\Gamma$ with a finite-index subgroup $\Gamma_0$ such that $\hat{\rho}(\Gamma_0)$ is a subgroup of $\mathsf{SL}(4q+2,\mathbb{R})$. Notice that $\hat{\rho}$ has to be faithful since $\xi$ is $\hat{\rho}$-equivariant and dynamics preserving for $\hat{\rho}$. \par 
Let $V_q=\wedge^{2q+1}\mathbb{R}^{4q+2}$, and notice by assumption that $\xi_{q}=\tau_{2q+1}^{+}\circ \xi$ is $\wedge^{2k+1}\rho$-equivariant and dynamics preserving. We consider the following two cases:

\par \emph{Case 1.} Suppose that $\Gamma$ has infinitely many ends. Then we show that $\Gamma$ is free. If not, by Stallings' theorem \cite{Stallings}, there exists a splitting $\Gamma=\Gamma_1\ast...\ast\Gamma_k \ast F_{s}$, where $s \geqslant 0$ and for $1 \leqslant i \leqslant k$, $\Gamma_i$ is an one-ended word hyperbolic group. In particular, there exists a quasiconvex subgroup of $\Gamma$ of the form $\Delta \ast \mathbb{Z}$, with $\Delta$ one-ended. Lemma \ref{free}, shows that there exists a quasiconvex free subgroup $H_0$ of $\Delta_2$ such that $\wedge^{2q+1}\rho(H_0)$ is $P_1$-Anosov in $\mathsf{SL}(V_q)$ and its limit map is the restriction $\xi_{q}:\partial_{\infty}H_0 \rightarrow \mathbb{P}(V_q)$. \par Since $\wedge^{2q+1}\rho(\delta)$ is proximal for every $\delta \in H_0 \subset \Delta_2$, by Lemma \ref{positive1}, $\ell_1(\wedge^{2q+1}(\rho(\delta)))>0$. The representation $\rho:H_0 \rightarrow \mathsf{SL}(4q+2,\mathbb{R})$ is $P_{2q+1}$-Anosov and $\wedge^{2q+1}\rho(\gamma)$ is positively proximal for every non-trivial $\gamma \in H_0$. By Theorem \ref{mainproperties} \textup{(iii)}, we can find a path connected open neighbourhood $U$ of $\rho_{0}:=\rho|_{H_{0}}$ in $\textup{Hom}(H_0,\mathsf{SL}(4q+2,\mathbb{R}))$ consisting of entirely of $P_{2q+1}$-Anosov representations. Proposition \ref{density} guarantees that there exists $\rho_{1} \in U$ such that $\rho_{1}(F_k)$ is Zariski dense in $\mathsf{SL}(4q+2,\mathbb{R})$. Let $\{\rho_{t}\}_{0\leqslant t\leqslant 1}$ be a continuous path between $\rho_{0}$ and $\rho_{1}$ contained entirely in $U$. By Fact \ref{continuity}, for every $\gamma \in H_0$, the map $t \mapsto \ell_1(\wedge^{2q+1}\rho_{t}(\gamma))$ is continuous with real values and nowhere vanishing. Hence $\ell_1(\wedge^{2q+1}\rho_{1}(\gamma))>0$ for every $\gamma  \in H_0$. Therefore, since $\wedge^{2k+1}$ is an irreducible representation, the group $\wedge^{2q+1}\rho_1(H_0)$ is a strongly irreducible subgroup of $\mathsf{SL}(V_q)$ which is positively proximal. By Theorem \ref{positivity3}, the group $\wedge^{2q+1}\rho_{1}(H_0)$ preserves a properly convex cone and hence a properly convex domain of $\mathbb{P}(V_q)$. On the other hand, the group $\wedge^{2q+1}\mathsf{SL}(4q+2,\mathbb{R})$ (and hence $\wedge^{2q+1}\rho_{1}(H_0)$) preserves the symplectic non-degenerate form $\omega_{q}:V_q \times V_q \rightarrow \mathbb{R}$ given by the formula $\omega_{q}(a,b)=a\wedge b \in \langle e_1 \wedge...\wedge e_{4q+2} \rangle$. However, by \cite[Corollary 3.5]{benoist-divisible0}, a strongly irreducible subgroup of $\mathsf{SL}(d, \mathbb{R})$ which preserves a symplectic form cannot preserve a properly convex domain of $\mathbb{P}(\mathbb{R}^d)$. We have reached a contradiction, so $\Gamma$ cannot contain any non-trivial one-ended factors in its free product decomposition. Therefore, $\Gamma$ is free.

\par \emph{Case 2.} Suppose that $\Gamma$ is one-ended and not virtually a surface group. Wilton's result \cite[Corollary B]{Wilton} ensures that $\Gamma$ contains a quasiconvex subgroup $\Delta$ which is either isomorphic to a surface group or rigid. If $\Delta$ has infinite index in $\Gamma$, then there exists a quasiconvex subgroup of $\Gamma$ isomorphic to $\Delta \ast \mathbb{Z}$. However, by the previous case we obtain a contradiction. Therefore, we may assume that $\Delta$ is rigid and has finite index in $\Gamma$. By Lemma \ref{free}, there exists $H_1$ a quasiconvex free subgroup of $\Delta_2$ such that the restriction $\wedge^{2q+1}\rho|_{H_1}$ is $P_1$-Anosov. By Lemma \ref{positive2}, for every $h \in H_1$, $\wedge^{2q+1}\rho(h)$ is positively proximal in $\mathbb{P}(V_q)$. By continuing as previously, we obtain a $P_{2q+1}$-Anosov, Zariski dense deformation $\rho_1$ of $\rho|_{H_1}$ such that $\wedge^{2q+1}\rho_1(H_{k})$ is positively proximal. Again, by Theorem \ref{positivity3}, $\wedge^{2q+1}\rho_1(H_{k})$ preserves a properly convex domain and the symplectic form $\omega_{q}$, a contradiction.
\par If $\rho$ is not faithful, Proposition \ref{discrete} shows that $\textup{ker}(\rho)$ is finite. The group $\Gamma'=\Gamma/ \textup{ker}\rho$ is word hyperbolic, $\partial_{\infty}\Gamma'=\partial_{\infty}\Gamma$, so $\xi$ is a $\rho'$-equivariant dynamics preserving map, where $\rho':\Gamma' \rightarrow \mathsf{GL}(4q+2,\mathbb{R})$ is the faithful representation induced by $\rho$. By Selberg's lemma, there exists a torsion free finite-index subgroup $\Gamma_1$ of $\Gamma'$. The previous arguments imply that $\Gamma_1$ is either a surface group or a free group. Therefore, $\Gamma$ is either a finite extension of a virtually free group or a virtually surface group. In the second case, its boundary is the circle and by \cite{Gabai}, $\Gamma$ is virtually a surface group. In the first case, by \cite{Dunwoody}, $\Gamma$ has infinitely many ends and splits as the fundamental group of a finite graph of groups with finite edge groups and vertex groups of at most one end. The vertex groups of this splitting are also finite extensions of a virtually free group hence finite. It follows that $\Gamma$ is virtually free.\end{proof}

By following the argument of case 1 in of the proof of Theorem \ref{maintheorem} we obtain the following conclusion:

\begin{theorem}\label{positive} Let $F_2$ be the free group on two generators and $\rho:F_2 \rightarrow \mathsf{GL}(4q+2,\mathbb{R})$ a representation. Suppose that $\rho$ is $P_{2q+1}$-Anosov. Then $\wedge^{2q+1}\rho(F_2)$ is not a positively proximal subgroup of $\mathsf{GL}(\wedge^{2q+1}\mathbb{R}^{4q+2})$.\end{theorem}
\medskip

For the proof of Corollary \ref{cor1} we need the following proposition for the existence of lifts of \hbox{$P_{2k+1}$-Anosov} representations into $\mathsf{PGL}(d,\mathbb{R})$. The proof is similar to Lemma \ref{positive1} and \ref{positive2}. In the case $\rho$ is irreducible and $k=0$, Zimmer has proved the existence of lifts in \cite[Theorem 3.1]{Zimmer}.
\medskip

\begin{proposition} Let $\Gamma$ be a torsion free word hyperbolic group and $\rho:\Gamma \rightarrow \mathsf{PGL}(d, \mathbb{R})$ is a $P_{2k+1}$-Anosov representation, where $0 \leqslant k \leqslant \frac{d-1}{4}$. 
\medskip

\noindent \textup{(i)} Suppose that $\Delta$ is an infinite index, one-ended quasiconvex subgroup of $\Gamma$ and $\rho_{0}$ is the restriction of $\rho$ on $\Delta$. There exists a lift $\widetilde{\rho_0}:\Delta \rightarrow \mathsf{GL}(d, \mathbb{R})$ such that $\wedge^{2k+1}\widetilde{\rho_0}(\Delta)$ is positively proximal.\\

\noindent \textup{(ii)} If $\Gamma$ is a rigid word hyperbolic group then there exists a lift $\widetilde{\rho}:\Gamma \rightarrow \mathsf{GL}(d, \mathbb{R})$ of $\rho$ such that $\wedge^{2k+1}\rho(\Gamma)$ is positively proximal.\end{proposition}
\medskip

\begin{proof}  We begin with the following observation: suppose that $\varphi:\Gamma \rightarrow \mathsf{PGL}(V_1\oplus V_2)$ is a representation such that $\varphi(\gamma)$ preserves $V_1$ for every $\gamma \in \Gamma$. If $\rho(\gamma)=[g_{\gamma}]$ then the map $\varphi_{0}(\gamma)=[g_{\gamma}|_{V_1}]$ is a well defined representation $\varphi_{0}:\Gamma \rightarrow \mathsf{PGL}(V_1)$. If $\varphi_0$ admits a lift $\widetilde{\varphi}_{0}$, then there exists a lift $\widetilde{\varphi}$ of $\varphi$ such that $\widetilde{\varphi}(\gamma)|_{V_1}=\widetilde{\varphi}_{0}(\gamma)$ for every $\gamma \in \Gamma$. The lift $\widetilde{\varphi}$ is defined as follows: for $\gamma \in \Gamma$, $\widetilde{\varphi}(\gamma)$ is the unique element $h_{\gamma} \in \mathsf{GL}(V_1\oplus V_2)$ such that the restriction of $h_{\gamma}$ on $V_1$ is $\widetilde{\varphi}_{0}(\gamma)$ and $\varphi(\gamma)=[h_{\gamma}]$. 
\par  Notice that we may asssume that $k=0$, because the exterior power $\wedge^{2k+1}:\mathsf{GL}(d,\mathbb{R})\rightarrow \mathsf{GL}(\wedge^{2k+1}\mathbb{R}^{d})$ is faithful. For part \textup{(i)}, we may consider $\delta \in \Gamma$ with $\delta^{\pm} \notin \partial_{\infty}\Delta$  and $\xi(\partial_{\infty}\Delta)$ is a connected compact subset of the affine chart $\mathbb{P}(\mathbb{R}^{d})-\mathbb{P}(V_{\rho(\delta)}^{-})$. In particular, $\xi(\partial_{\infty}\Delta)$ lies in the affine chart $A=\mathbb{P}(V)-\mathbb{P}(V\cap V_{\rho(\delta)}^{-})$ of $\mathbb{P}(V)$, where $V=\langle \xi(\partial_{\infty}\Delta) \rangle$. Since $\rho_{0}(\Delta)$ preserves $V$ there exists a well defined representation $\rho_{1}:\Delta \rightarrow \mathsf{PGL}(V)$. The image $\rho_{1}(\Delta)$ preserves the connected compact set $\xi(\partial_{\infty}\Delta)$ and hence the interior of the convex hull of $\xi(\partial_{\infty}\Delta)$ in $A$. There exists a lift $\widetilde{\rho_{1}}$ of $\rho_{1}$ into $\mathsf{GL}(V)$ such that $\widetilde{\rho_1}(\Delta)$ preserves a properly convex cone $C$ of $V$. The representation $\widetilde{\rho_1}$ is $P_1$-Anosov, faithful and by Lemma \ref{cone}, $\widetilde{\rho_1}(\gamma)$ is positively proximal for every $\gamma \in \Delta$ non-trivial. By our initial observation we obtain a lift $\widetilde{\rho_{0}}:\Delta \rightarrow \mathsf{GL}(d,\mathbb{R})$ of $\rho_{0}$ with $\widetilde{\rho_{0}}(\gamma)|_{V}=\widetilde{\rho_1}(\gamma)$. The representation $\widetilde{\rho_1}$ is $P_{1}$-Anosov with Anosov limit map $\xi$. For every non-trivial $\gamma \in \Delta$, the attracting fixed point of $\widetilde{\rho_0}(\gamma)$ is in $V$ and $\ell_1(\widetilde{\rho_0}(\gamma))=\ell_1(\widetilde{\rho_1}(\gamma))>0$.
\par The proof of \textup{(ii)} follows by observing, as in Lemma \ref{positive2}, that the image of $\partial_{\infty}\Gamma$ under the Anosov limit map $\xi$ lies in an affine chart of $\mathbb{P}(\mathbb{R}^d)$. Then we continue as previously to obtain the lift $\widetilde{\rho}$. \end{proof}
\medskip

\noindent \emph{Proof of Corollary \ref{cor1}}. We assume that $\Gamma$ is torsion free. If $\Gamma$ contains a quasiconvex infinite index one-ended subgroup $\Gamma_0$, there exists a lift $\widetilde{\rho}_0$ of $\rho|_{\Gamma_0}$ such that the group $\wedge^{2k+1}\widetilde{\rho}_{0}(\Gamma_0)$ is positively proximal, contradicting Theorem \ref{positive}. Also $\Gamma$ cannot be rigid again by part (ii) of the previous proposition. Therefore, $\Gamma$ is either free or has one end and by \cite[Corollary B]{Wilton} there exists a quasiconvex surface subgroup which has to be of finite index in $\Gamma$. In every case, since $\textup{ker}(\rho)$ is finite, the boundary $\partial_{\infty}\Gamma$ is either a circle or totally disconnected so $\Gamma$ is virtually free or virtually a surface group. $ \ \ \square$
\medskip

\noindent \emph{Proof of Corollary \ref{rigid}}. Suppose that there exists a continuous $\rho$-equivariant map $\xi$ and $\rho(\gamma) \in \rho(\Gamma)$ a $P_{2q+1}$-proximal element with $\xi(\gamma^{+})=x_{\rho(\gamma)}^{+}$ and $\xi(\gamma^{-})=x_{\rho(\gamma)}^{-}$. The map $\xi^{+}:=\tau_{2q+1}^{+}\circ \xi$ is $\wedge^{2q+1}\rho$-equivariant and by Lemma \ref{free} there exist a free quasiconvex subgroup $H$ of $\Gamma_2$ such that $\wedge^{2q+1}\rho|_{H}$ is $P_1$-Anosov. Lemma \ref{positive2} shows that $\wedge^{2q+1}\rho(H)$ is positively proximal, a contradiction by Theorem \ref{positive}.
\par Let $V_q=\wedge^{2q+1}\mathbb{R}^{4q+2}$ and $\xi^{-}=\tau^{-}_{2q+1}\circ \xi$. We show that the map $\xi^{+}$ cannot be spanning. Suppose that $\xi^{+}$ is spanning and $x_1,...,x_r \in \partial_{\infty}\Gamma$ with $V_q=\oplus_{i=1}^{r}\xi^{+}(x_i)$, $r=\textup{dim}(V_q)$. Since $\Gamma$ acts minimally on $\partial_{\infty}\Gamma$, every open subset $U$ of $\partial_{\infty}\Gamma$, $\xi^{+}(U)$ spans $V_q$ and the union $\cup_{i=1}^{r}\xi^{-}(x_i)$ cannot contain $\xi^{+}(\partial_{\infty}\Gamma)$. There exists $y \in \partial_{\infty}\Gamma$ and $1 \leqslant j \leqslant r$ with $V_q=\xi^{+}(x_j)\oplus \xi^{-}(y)=\xi^{+}(y)\oplus \xi^{-}(x_j)$. By the density of pairs $\{ (\delta^{+},\delta^{-}): \delta \in \Gamma\}$ in the set of $2$-tuples of $\partial_{\infty}\Gamma$, we can find $\gamma \in \Gamma$ such that $V_q=\xi(\gamma^{+})\oplus \xi^{-}(\gamma^{-})=\xi^{+}(\gamma^{-})\oplus \xi^{-}(\gamma^{+})$. 
\par Then we claim that $g=\wedge^{2q+1}\rho(\gamma)$ is a biproximal matrix. Up to conjugating $g$ we may assume that $\xi^{+}(\gamma^{+})=[e_1]$, $\xi^{-}(\gamma^{-})=[e_{1}^{\perp}]$ and write $g=\begin{bmatrix}
a(g) & 0\\ 
0 & A
\end{bmatrix}$ for some matrix $A \in \mathsf{GL}(e_1^{\perp})$. Suppose that $\lambda_1(A)\geqslant |a(g)|$. Let $p\geqslant 1$ be the largest possible dimension of a complex Jordan block corresponding to an eigenvalue of maximum modulus of $A$. Then there exists a subsequence $(k_n)_{n \in \mathbb{N}}$, $A_{\infty}$ a non-zero matrix and $b \in \mathbb{R}$ with $$\lim_{n \rightarrow \infty}\frac{1}{k_{n}^{p-1}\lambda_{1}(A)^{k_n}}g^{k_n}=\begin{bmatrix}
b & 0\\ 
0 & A_{\infty}
\end{bmatrix}$$ Since $\partial_{\infty}\Gamma$ is perfect and $\xi^{+}(\partial_{\infty}\Gamma)$ spans $V_q$, we may choose $x \in \partial_{\infty}\Gamma-\{ \gamma^{-} \}$ such that the projection of $\xi^{+}(x)$ in $e_{1}^{\perp}$ is not in $\textup{ker}(A_{\infty})$. Thus, $\lim_{n} g^{k_n}\xi^{+}(x)=\lim_{n}\xi^{+}(\gamma^{k_n}x)=\xi^{+}(\gamma^{+})$ cannot be the line $[e_1]$, a contradiction. It follows that $|a(g)|>\lambda_1(A)$ and $\wedge^{2q+1}\rho(\gamma)$ is proximal with attracting fixed point $\xi^{+}(\gamma^{+})$. Since $V_q=\xi^{+}(\gamma^{-})\oplus \xi^{-}(\gamma^{+})$, the same argument shows that $\wedge^{2q+1}\rho(\gamma^{-1})$ is proximal with attracting fixed point $\xi^{+}(\gamma^{-})$. The map $\xi^{+}$ (and hence $\xi$) preserves the dynamics of $\{\gamma^{-},\gamma^{+}\}$. This contradicts the fact that $\xi$ is nowhere dynamics preserving. Therefore, $\tau_{2q+1}^{+}(\xi(\partial_{\infty}\Gamma))$ lies in some proper vector subspace of $V_q$. $\ \ \square$\\

\section{Examples}\label{examples}
In this section we provide an example showing that the analogue of Theorem \ref{maintheorem} does not hold in dimensions which are multiples of $4$. Also, we give an example of a surface group representation $\rho$ into $\mathsf{SL}(4q+2,\mathbb{R})$ which is not $P_{2q+1}$-Anosov but admits a $\rho$-equivariant continuous dynamics preserving map $\xi$ into $\mathsf{Gr}_{2q+1}(\mathbb{R}^{4q+2})$. Let $S$ be a closed orientable hyperbolic surface and $\tau_{2}:\mathsf{SL}(2,\mathbb{C}) \rightarrow \mathsf{SL}(4,\mathbb{R})$ be the standard inclusion defined as $\tau_{2}(g)=\begin{bmatrix}
\textup{Re}(g) &-\textup{Im}(g) \\ 
\textup{Im}(g) & \textup{Re}(g)
\end{bmatrix}$ for $g \in \mathsf{SL}(2,\mathbb{C})$.

\medskip

\begin{Example} \normalfont Let $F_2$ be the free group on two generators. The group $\Gamma=\pi_1(S)\ast F_{2}$ admits an Anosov representation $\rho$ into $\mathsf{SL}(2,\mathbb{C})$ and hence $\tau_2\circ \rho$ is a $P_{2}$-Anosov representation into $\mathsf{SL}(4,\mathbb{R})$. For $k \in \mathbb{N}$, the representation $\rho_{k}=\oplus_{i=1}^{k}(\tau_{2}\circ \rho)$ of $\Gamma$ into $\mathsf{SL}(4k,\mathbb{R})$ is $P_{2k}$-Anosov. In fact, by Theorem \ref{mainproperties} (iii) and Proposition \ref{density} there exists a deformation $\rho_{k}'$ of $\rho_{k}$ which is Zariski dense and $P_{2k}$-Anosov. \end{Example}
\medskip

\begin{Example} \normalfont Let $M$ be the mapping torus of the closed hyperbolic surface $S$ with respect to a  fixed pseudo-Anosov homeomorphism $\phi: S \rightarrow S$. The group $\pi_1(M)$ contains a normal infinite index subgroup $\Gamma$ isomorphic with $\pi_1(S)$. By a theorem of Thurston \cite{Thurston} (see also Otal \cite{Otal}), the group $\pi_1(M)$ admits a convex cocompact representation $\iota$ into $\mathsf{PSL}(2,\mathbb{C})$. In fact, by \cite{Culler}, $\iota$ lifts to a quasi-isometric embedding $\widetilde{\iota}:\pi_1(M)\rightarrow \mathsf{SL}(2,\mathbb{C})$. By composing  $\tau_2$ with $\widetilde{\iota}$, we obtain a $P_2$-Anosov representation $\rho_1: \pi_1(M) \rightarrow \mathsf{SL}(4,\mathbb{R})$. The Cannon-Thurston map (see \cite{CannonThurston}), $\theta: \partial_{\infty}\pi_1(S)\rightarrow \partial_{\infty}\pi_1(M)$ composed with the Anosov limit map $\xi^{2}_{\rho_1}:\partial_{\infty}\pi_1(M) \rightarrow \mathsf{Gr}_{2}(\mathbb{R}^4)$ provides a $\rho_1|_{\Gamma}$-equivariant dynamics preserving map $\xi_0:\partial_{\infty}\Gamma \rightarrow \mathsf{Gr}_{2}(\mathbb{R}^4)$. Note that the representation $\rho_1|_{\Gamma}$ is not a quasi-isometric embedding, in particular not $P_2$-Anosov, since $\Gamma$ is not a quasiconvex subgroup of $\pi_1(M)$. Let $\rho_{F}:\Gamma \rightarrow \mathsf{SL}(2,\mathbb{R})$ be a Fuchsian representation with limit map $\xi_{\rho_F}^{1}$. The representation $\rho=(\oplus_{i=1}^{q} \rho_{1}|_{\Gamma}) \oplus \rho_{F}$ into $\mathsf{SL}(4q+2,\mathbb{R})$ is not $P_{2q+1}$-Anosov, however the $\rho$-equivariant map $\xi=(\oplus_{i=1}^{r}\xi_0)\oplus \xi_{\rho_F}^{1}$ is dynamics preserving.  \end{Example}


\vspace{0.4cm}


\begin{thebibliography}{100}
\bibitem{Barbot} T. Barbot, ``Three-dimensional Anosov flag manifolds,'' {\em Geom. Top.} {\bf 14}(2010), 153-191. 

\bibitem{benoist-divisible0} Y. Benoist, ``Automorphismes des c\^ ones convexes,'' {\em Invent. Math.} {\bf 141}(2000) 149-193.

\bibitem{benoist-divisible3} Y. Benoist, ``Convexes divisibles III,'' {\em Ann. Sci. de l' \'E.N.S} {\bf 38}(2005), 793-832.

\bibitem{BPS} J. Bochi, R. Potrie and A. Sambarino, ``Anosov representations and dominated splittings,''
{\em J. Eur. Math. Soc.}, to appear, arXiv:1605.01742.

\bibitem{Bowditch} B. Bowditch, ``Cut points and canonical splittings of hyperbolic groups,'' {\em Acta Math.} {\bf 180}(1998), 145-186. 

\bibitem{BGGT} Breuillard E., Green B., Guralnick R., Tao T., ``Strongly dense free subgroups of semisimple algebraic groups,'' {\em Israel J. Math.} {\bf 192}(2012), 347-379.

\bibitem{Culler} Culler M., ``Lifting representations to covering groups,'' {\em Adv. Math.} {\bf 59}(1986), 64-70.

\bibitem{Dunwoody} M. J. Dunwoody, ``The accessibility of finitely presented groups,'' {\em Invent. Math.} {\bf 81}(1985), 449-457.

\bibitem{Gabai} D. Gabai, ``Convergence groups are Fuchsian groups,'' {\em Ann. Math.} {\bf 136}(1992), 447-510.

\bibitem{CLS} R. Canary, M. Lee, M. Stover, ``Amalgam Anosov representations,''  With an appendix
by R. Canary, M. Lee, A. Sambarino and M. Stover, {\em Geom. Top.} {\bf 21}(2017), 215-251.

\bibitem{CT} R. Canary and K. Tsouvalas, ``Topological restrictions on Anosov representations,'' preprint, arXiv:1904.02002.

\bibitem{CannonThurston} J. Cannon and W. Thurston, ``Group invariant Peano curves,''  {\em Geom. Topol.} {\bf 11}(2007), 1315--1355.

\bibitem{GGKW} F. Gu\'eritaud, O. Guichard, F. Kassel and A. Wienhard, ``Anosov representations and proper actions,'' {\em Geom. Top.} {\bf 21}(2017), 485--584.

\bibitem{GW} O. Guichard and A. Wienhard, ``Anosov representations: Domains of discontinuity and
applications,'' {\em Invent. Math.} {\bf 190}(2012), 357--438.

\bibitem{KLP1} M. Kapovich, B. Leeb and J. Porti, ``A Morse Lemma for quasigeodesics in symmetric spaces and 
Euclidean buildings,'' {\em Geom. Top} {\bf 22}(2018), 3827--3923. 

\bibitem{KLP2} M. Kapovich, B. Leeb and J. Porti, ``Morse actions of discrete groups on symmetric spaces,'' preprint, arXiv:1403.7671.

\bibitem{KLP3} M. Kapovich, B. Leeb and J. Porti, ``Some result results on Anosov representations,'' {\em Transform. Groups} {\bf 4}(2016), 1105-1121.

\bibitem{KP} F. Kassel and R. Potrie, ``Eigenvalue gaps for hyperbolic groups and semigroups,'' in preparation.

\bibitem{labourie-invent} F. Labourie, ``Anosov flows, surface groups and curves in projective space,''
{\em Invent. Math.} {\bf 165}(2006), 51--114.

\bibitem{Otal} J. P. Otal, ``Le th\'eor$\grave{\textup{e}}$me d'hyperbolisation pour les vari\'et\'es fibre\'es de dimension 3,'' {\em Ast\'erisque} {\bf 235}(1996)

\bibitem{Stallings} J. R. Stallings, ``On torsion free groups with infinitely many ends,'' {\em Ann. of Math.} {\bf 88}
(1968), 312-334.

\bibitem{Thurston} W. Thurston, ``Hyperbolic structures on 3-manifolds, II: Surface groups and 3-manifolds which fiber over the circle,'' preprint, arXiv:math/9801045, 1998.

\bibitem{Wilton} H. Wilton, ``Essential surfaces in graph pairs,'' {\em Journal of the A.M.S.} {\bf 31}(2018), 893-919.

\bibitem{Zimmer} A. Zimmer, ``Projective Anosov representations, convex cocompact actions, and rigidity,''
preprint, arXiv:1704.08582.

\end{thebibliography}
\end{document}